\documentclass[14pt]{article}
\usepackage{amssymb,amsmath,amsthm,mathrsfs}
\usepackage[usenames]{color}
\usepackage{hyperref}
\usepackage[all]{xy}
\usepackage{xypic}
\usepackage{graphicx}
\usepackage{amscd}
\usepackage{verbatim}
\usepackage{enumerate}
\usepackage{subcaption}
\usepackage[shortlabels]{enumitem}
\usepackage{tikz-cd} 
\usepackage{mathtools} 

\baselineskip 18pt \textwidth 16cm  \theoremstyle{plain}

\renewcommand{\labelenumi}{(\arabic{enumi})}
\renewcommand{\labelenumii}{(\arabic{enumi}\alph{enumii})}
\renewcommand{\labelenumiii}{\arabic{enumi}.\alph{enumii}.\roman{enumiii}}

\newcommand{\chen}[1]{{\color{blue}{Chen: #1}}}
\newcommand{\nir}[1]{{\color{red}{Nir: #1}}}

\makeatletter
\newcommand*{\rom}[1]{\expandafter\@slowromancap\romannumeral #1@}
\makeatother

\newtheorem{theorem}{Theorem}[subsection]

\newtheorem{corollary}[theorem]{Corollary}
\newtheorem{lemma}[theorem]{Lemma}
\newtheorem*{lemma*}{Lemma}
\newtheorem{claim}[theorem]{Claim}

\theoremstyle{definition}
\newtheorem{definition}[theorem]{Definition}
\newtheorem{notation}[theorem]{Notation}

\theoremstyle{remark}
\newtheorem{remark}[theorem]{Remark}

\newcounter{mycounter}

\def\A{\Bbb A}
\def\C{\Bbb C}
\def\N{\Bbb N}
\def\O{\Bbb O}
\def\Q{\Bbb Q}
\def\R{\Bbb R}
\def\Z{\Bbb Z}
\def\NU{{N^*}}
\def\QU{K^*}
\def\ZU{O^*}
\def\ZZU{{Z^*}}
\def\ZC{\Bbb O}
\def\QC{\Bbb K}
\def\IC{\Bbb I}
\def\JC{\Bbb J}
\def\ZI{Z^*_\infty}
\def\NI{N^*_\infty}
\def\pr{\ell}
\def\E{\mathrm{E}^{\nom}}
\def\Com{\mathrm{Com}}

\DeclareMathOperator{\ari}{ar}
\DeclareMathOperator{\Cent}{C}
\DeclareMathOperator{\ccl}{gcl}
\DeclareMathOperator{\Der}{D}
\DeclareMathOperator{\diag}{diag}
\DeclareMathOperator{\disc}{disc}
\DeclareMathOperator{\dist}{dist}
\DeclareMathOperator{\Id}{Id}
\DeclareMathOperator{\id}{id}
\DeclareMathOperator{\Ima}{Im}
\DeclareMathOperator{\Isom}{Isom}
\DeclareMathOperator{\Mat}{Mat}
\DeclareMathOperator{\Mod}{mod}
\DeclareMathOperator{\G}{G}
\DeclareMathOperator{\gcl}{gcl}
\DeclareMathOperator{\GL}{ GL}
\DeclareMathOperator{\Mob}{Mob}
\DeclareMathOperator{\Or}{O}
\DeclareMathOperator{\PGL}{PGL}
\DeclareMathOperator{\PSL}{PSL}
\DeclareMathOperator{\rank}{rank}
\DeclareMathOperator{\Rea}{Re}
\DeclareMathOperator{\SL}{SL}
\DeclareMathOperator{\SO}{SO}
\DeclareMathOperator{\Sp}{Sp}
\DeclareMathOperator{\Spin}{Spin}
\DeclareMathOperator{\Sym}{Sym}
\DeclareMathOperator{\SU}{SU}
\DeclareMathOperator{\supp}{supp}
\DeclareMathOperator{\Span}{Span}
\DeclareMathOperator{\Spec}{Spec}
\DeclareMathOperator{\Stab}{Stab}
\DeclareMathOperator{\val}{val}
\DeclareMathOperator{\width}{wid}
\DeclareMathOperator{\Co}{Com}
\DeclareMathOperator{\El}{E}

\title{Bounded Generation for $\SL_n(\Lambda)$}
\author{Nir Avni and Chen Meiri}
\begin{document}
\maketitle

\begin{abstract}
Let $\Lambda$ be an order in a division algebra over a number field. We prove, under some conditions, that $\SL_3(\Lambda)$ is boundedly generated by elementary matrices.
\end{abstract}
	
\section{Introduction}

The width of a subset $X$ of a group $\Gamma$ is the minimal $n \in \mathbb N \cup \left\{\infty \right\}$ for which every element in the group generated by $X$ can be written as a product of at most $n$ elements of $X$ or their inverses. Width questions were studied for many types of groups and subsets, see e.g.,  \cite{BBF19}, \cite{Bor83}, \cite{DV88}, \cite{LaSh09}, \cite{NiSe07a} and \cite{NiSe07b}  and the references therein.

In this article we explain a new approach to width questions in arithmetic groups. We used this approach in \cite{AM} to prove that, under the Generalized Riemann Hypothesis, conjugacy classes have finite width in many higher rank anisotropic groups of orthogonal type. As far as we know, this is the first width result for anisotropic groups. Here, we use a simplified version of the method to prove, unconditionally, new cases of bounded generation.

A group $\Gamma$ is called boundedly generated if there exist finitely many cyclic subgroups whose union has finite width. Carter and Keller proved in \cite{CK83} that if $O$ is the ring of integers in a number field and $n \geq 3$ then the set of elementary matrices in $\SL_n(O)$ has finite width. Thus, $\SL_n(O)$ is boundedly generated. This was extended to rings of $S$-integers in Chevalley groups of higher rank in \cite{Tav91} and to many orthogonal groups in \cite{ER06}. It was expected that all higher rank lattices are boundedly generated, but, surprisingly, \cite{CRRZ22} proved that anisotropic lattices are {\em never} boundedly generated. It is still believed that higher rank isotropic groups are boundedly generated; in this note, we prove bounded generation for a new family of isotropic higher rank arithmetic groups, namely, groups of the form $\SL_d(\Lambda)$, where $d\ge 3$ and $\Lambda$ is an $S$-order in a central division algebra over a number field:

\begin{theorem} \label{thm:main} Let $d\ge 3$, $F$ be a number field, $S$ be a finite set of places of $F$, and $D$ be an $n^2$-dimensional central division algebra over $F$. Let $O$ be the ring of $S$-integers of $F$ and let $\Lambda \subseteq D$ be an $O$-order. Assume that:
\begin{enumerate}[(1)]
\item  $\SL_3(\Lambda)$ is generated by the elementary matrices $\left\{ e_{i,j}(\lambda) \mid \lambda \in \Lambda \right\}$. In particular,  $\SL_3(\Lambda)$ is prefect. 
\item $S$ contains all the archimedean places and at least one finite place.   
\item $\Lambda$ is a free module over $O$.
\item For every place $v \not\in S$, there exists an isomorphism of $O_v$-algebras $\psi_v:\Lambda_v \rightarrow M_n(O_v)$.
\end{enumerate}
Then $\SL_d(\Lambda)$ is boundedly generated by the elementary matrices.
\end{theorem} 

\begin{remark} We will show that it is enough to prove Theorem \ref{thm:main} in the case $d=3$. In the rest of the paper, we mostly consider this case.
\end{remark}

\begin{remark} Let $\overline{F}:=\prod'_{v \notin S } F_v$ be the $S$-congruence completion of $F$ and, similarly, let $\overline{\Lambda}$ and $\overline{D}$ be the $S$-congruence completions of $\Lambda$ and $D$. Under the assumptions above, the isomorphisms $\psi_v$ extend uniquely to isomorphisms $\psi_v:D_v \rightarrow M_n(F_v)$ of $F_v$-algebras and the map $\psi=(\psi_v)_{v \not \in S}:\overline{D}\rightarrow M_n(\overline{F})$ is an isomorphism of $\overline{F}$-algebras such that $\psi(\overline{\Lambda})=M_n(\overline{O})$.
\end{remark} 

{\bf Sketch of the proof:} From now on fix $F,S,D,O,\Lambda$ as in Theorem \ref{thm:main}. In order to prove the theorem, we fix a non-principal ultrafilter and consider the ultrapower $\SL_3(D)^*$ of $\SL_3(D)$. The group $\SL_3(D)^*$ is a topological group with respect to its congruence topology, which extends the congruence topology of $\SL_3(\Lambda)^*$. Denoting the congruence completion of $\SL_3(D)^*$ by $\overline{\SL_3(D)^*}$, we show that if $\SL_3(\Lambda)$ fails to be boundedly generated, then $\overline{\SL_3(D)^*}$ has a non-trivial discrete metaplectic extension, i.e., a non-trivial topological central extension
\[
1 \rightarrow C \rightarrow E \rightarrow \overline{\SL_3(D)^*} \rightarrow 1,
\]
with a discrete kernel, a continuous splitting over $\overline{\SL_3(\Lambda)^*}$, and an abstract splitting over $\SL_3(D)^*$. To finish the proof, we imitate  the computation of the metaplectic kernel of the countable group $\SL_3(D)$ and show that $\overline{\SL_3(D)^*}$ does not have a non-trivial metaplectic extension.

\begin{notation} Let $G$ be a group. For $X,Y \subseteq G$ we denote $[X,Y]:=\{[x,y] \mid x \in X,y \in Y\}$. 
\end{notation}

\subsection{Acknowledgements} 
%We thank Danny Neftin, Eugene Plotkin, Pavel Gvozdevsky, Peter Sarnak, and V.N. Venkataramana for their help. 
N.A. was supported by NSF grant DMS-1902041, C.M. was supported by ISF grant 872/23 and BSF grant 2020119. We thank Andrei Rapinchuk for helpful discussions.

\section{Completions}\label{sec:Completions}

	\subsection{The congruence completion }\label{sec:the_congruence_completion}
	
	Let $U$ be a non-principal ultrafilter on $\N$.  For a sequence $(M_n)_n$ of algebraic structures, denote the ultraproduct of  $(M_n)_n$ over $U$ by $\prod_U M_n$. If all the $M_n$s are equal, say to $M$, we call $\prod_U M_n$ the ultrapower of $M$ with respect to $U$ and (sometimes) denote it by $M^*$. Given $(a_n)_n \in \prod_{n\in \N} M_n$, denote its image in $\prod_U M_n$ by $[a_n]_n$. Given a sequence $(A_n)_n$ such that $A_n \subseteq M_n$, for every $n$, denote 
    \[
    [A_n]_n:=\left\{[a_n]_n\in \sideset{}{_U}\prod M_n\quad \middle | \quad \{n \in \N\, |\, a_n \in A_n \}\in U\right\} .
    \]
	
	Denote the set of non-zero finitely generated ideals of $O^*$ by $\mathcal{I}$ and let $\mathcal{P} \subseteq \mathcal{I}$ be the set of finitely generated prime ideals. Since every ideal of $O$ is generated by two elements, a subset $A$ of $O^*$ is a finitely generated (prime) ideal if and only if there exists a sequence $(\mathfrak{q}_n)_n$  of (prime) ideals of	$O$ such that $A=[\mathfrak{q}_n]_n$.  In particular, every $\mathfrak{q} \in \mathcal{I}$   is generated by two elements. 

    Since every ideal contains a finitely generated ideal, the set $\mathcal{I}$ is a basis of open neighborhoods of zero for the congruence topology of $O^*$. Thus, the completion $\overline{O^*}$ of $O^*$ with respect to the congruence topology can be identified as
	\[	 \overline{O^*}=\underset{_{\mathfrak{q} \in \mathcal{I}}}{\lim_{\longleftarrow} }\,O^*/\mathfrak{q}.\]
 While inverse limits of non-compact spaces can be empty, even if all the maps are surjective, in our case, the map $a \mapsto (a+\mathfrak{q})_\mathfrak{q}$ is an embedding of $O^*$ in $\overline{O^*}$.  In fact,  $\prod_U \overline{O}$ is embedded in $\overline{O^*}$ where $\overline{O}$ is the congruence completion of $O$.
	Indeed, for every finitely generated ideal $\mathfrak{q}=[\mathfrak{q}_n]_n \lhd O^*$, 
	\begin{itemize}
		\item There exists a quotient  map $\rho_{\mathfrak{q}}:\prod_{U} \overline{O} \rightarrow \prod_U O/\mathfrak{q}_n$.
		\item  The map $\iota_{\mathfrak{q}}:\prod_U O/\mathfrak{q}_n \rightarrow O^*/\mathfrak{q}$ 	defined by  $\iota_{\mathfrak{q}}([a_n+\mathfrak{q}_n]_n ):= [a_n]_n+\mathfrak{q}$ is a bijection. 
	\end{itemize}
	It follows that the map $\iota:\prod_U \overline{O}\rightarrow \overline{O^*}$ defined by \[\iota([a_n]_n):=\big(\iota_{\mathfrak{q}}\left(\rho_{\mathfrak{q}}\left([a_n]_n\right)\right)\big)_{\mathfrak{q} \text{ is s finitely generated ideal of $O^*$} }\]
	is an embedding of $\prod_U \overline{O}$ in $\overline{O^*}$. We do not know whether this embedding is surjective.  From now on, we use this embedding to identify $\prod_U \overline{O}$ as a subring of $\overline{O^*}$. In particular, $\prod_U \overline{O}$ is a topological ring under the topology inherited from $ \overline{O^*}$ and, by the universal property of completion, $\overline{O^*}$ is the completion of $\prod_U \overline{O}$ with respect to this topology. 
	
	The congruence topology of $O^*$ can be extended to a {\bf ring} topology on $F^*$. We refer to this topology as the $S$-congruence topology of $F^*$ and denote the completion of $F^*$ with respect to the $S$-congruence topology by $\overline{F^*}$. As before, $\prod_U\overline{F}$ embeds as a subring of $\overline{F^*}$  and $\overline{F^*}$ is  the completion of $\prod_U\overline{F}$.

	\begin{definition} A set $P$ of primes ideals of $O^*$ is called an internal set of prime ideals (or internal set for short) if, for every $n$, there exists a set $P_n$ of non-zero prime ideals of $O$ such that \[P=[P_n]_n:=\{[\mathfrak{p}_n]_n \mid \{ n \mid \mathfrak{p}_n \in P_n\}\in U \}.\] If all the sets $P_n$ are nonempty then $[P_n]_n=\{[\mathfrak{p}_n]_n \mid (\forall n)\ \mathfrak{p}_n \in P_n\}$.
    If all the sets $P_n$ can be chosen to be finite, we say that $P$ is small internal. 
	\end{definition}
	
	If $P=[P_n]_n$ and $Q=[Q_n]_n$ are internal subsets then $P \cap Q=[P_n \cap Q_n]_n$ is an internal subset and $P^c:=[(P_n)^c]_n$ is an internal set where for every $n$, $(P_n)^c$ is the set of all non-zero prime ideals of $O$ that are not in $P_n$. 
    
    If $X$ is a set of non-zero prime ideals of $O$ let $\pi_X:\overline{F} \rightarrow \overline{F}$ be the projection to $\prod_{\mathfrak{p} \in X} F_\mathfrak{p} \times \prod_{\mathfrak{q} \notin X} 0$. Given a non-empty internal subset $P=[P_n]_n$, let $\pi_P:\prod_U \overline{F}\rightarrow \prod_U \overline{F}$ be the map
	$\pi_P([a_n]_n)=[\pi_{P_n}(a_n)]_n$. The map $\pi_P$ is continuous; denote its image by $(\prod_U \overline{F})_P$. If, in addition, $P\neq \mathcal{P}$ then the map $\pi_P \times \pi_{P^c}:\prod_U \overline{F} \rightarrow (\prod_U \overline{F})_P \times (\prod_U \overline{F})_{P^c}$ is an isomorphism of topological rings. \\ 
    
    Let $G \subseteq\GL_k$ be an affine group $F$-scheme. The $S$-congruence topology of $F^*$ induces a topology on $G(F^*)$, which is also called the $S$-congruence topology. Denote the completion of $G(F^*)$ under this topology by $\overline{G(F^*)}$. The embedding $G(F^*) \rightarrow G(\overline{F^*})$ extends to an injective continuous homomorphism $\overline{G(F^*)}\rightarrow G(\overline{F^*})$. 
    
\begin{lemma} (\cite[Corollary 3.6.5]{AM}) If $G(F)$ has Strong Approximation, then the map $\overline{G(F^*)} \rightarrow G(\overline{F^*})$ is an isomorphism of topological groups.
\end{lemma} 

When the lemma holds, we identify $\overline{G(F^*)}$ with $ G(\overline{F^*})$.

\begin{proof} To prove the lemma, it suffices to show that $G(F^*)$ is dense in $G(\overline{F^*})$, i.e., that for every $a\in G(\overline{F^*})$ and every $I\in \mathcal{I}$, the intersection $\left( a+I\Mat_k(\overline{O^*})\right) \cap G(F^*)$ is non-empty.

We first rephrase Strong Approximation in a way that passes to ultrapowers. Suppose that $G \subseteq \Mat_k$ is the zero locus of a polynomial map $p:\Mat_k \rightarrow \mathbb{A}^r$. For a subset $U \subseteq \Mat_k(F)$, say that $U$ almost contains a $G$-point if the congruence closure of $p(U)$ in $F^r$ contains 0. Since $\overline{O}$ is compact, the following conditions on a point $x\in \Mat_k(F)$ and an ideal $I \triangleleft O$ are equivalent: \begin{itemize} 
\item $x+I\Mat_k(O)$ almost contains a $G$-point.
\item $\left( x+I\Mat_k(\overline{O}) \right) \cap G(\overline{F})\neq \emptyset$.
\end{itemize}
Thus, Strong Approximation for $G$ is equivalent to the following first order claim:
\begin{enumerate}[(a)]
\item For all $x\in \Mat_k(F)$ and a non-zero finitely generated ideal $I\lhd O$, if $x+I\Mat_k(O)$ almost contains a $G$-point, then $\left( x+I\Mat_k(O)\right) \cap G(F)\neq \emptyset$.
\setcounter{mycounter}{\value{enumi}}
\end{enumerate}
By taking ultrapowers, we get the following claim:
\begin{enumerate}[(a)]
\setcounter{enumi}{\value{mycounter}}
\item \label{item:SAUP} For all $x\in \Mat_k(F^*)$ and a non-zero $I\in \mathcal{I}$, if $x+I\Mat_k(O^*)$ almost contains a $G$-point, then $\left( x+I\Mat_k(O^*)\right) \cap G(F^*)\neq \emptyset$.
\end{enumerate}

Finally, let $a\in G(\overline{F^*})$ and let $I\in \mathcal{I}$. Choose a point $b\in \left( a+I\Mat_k(\overline{O^*})\right) \cap \Mat_k(F^*)$. Since the set $b+I\Mat_k(O^*)$ contains $a$ in its closure, $b+I\Mat_k(O^*)$ almost contains a $G$-point. Thus, 
\[
\left( a+I\Mat_k(\overline{O^*})\right) \cap G(F^*)=\left( b+I\Mat_k(O^*)\right) \cap G(F^*)\neq \emptyset,
\]
where the right inequality follows from Item \ref{item:SAUP}.
\end{proof} 

    Given an affine group $F$-scheme $G \subseteq \GL_k$ such that $G(F)$ has Strong Approximation with respect to the $S$-congruence topology and a non-empty internal set $P$ of nonzore prime ideals, define $\pi_P:G(\prod_U \overline{F})\rightarrow G(\prod_U \overline{F}) $    by 
	\[
	\pi_P\left(\left( a_{i,j})_{1\le i,j \le k}\right)\right):=\left( \pi_P(a_{i,j}-\delta_{i,j})+\delta_{i,j}   
	\right)_{1\le i,j \le k},
	\]
	where $\delta_{i,j}$ is Kronecker's delta. As before, the image of $G(\prod_U \overline{F})$  under $\pi_P$ is denoted by $G(\prod_U \overline{F})_P$.  If $P$ and  $P^c$ are non-empty then the map $\pi_P \times \pi_{P^c}:G(\prod_U \overline{F}) \rightarrow G(\prod_U \overline{F})_P \times G(\prod_U \overline{F})_{P^c}$ is an isomorphism of topological groups.
	
	More generally, if $(H_n)_n$ is a sequence of affine group $F$-subschemes of $G$ of bounded complexity then 
	$H:=\prod_U H_n$ is an affine group  $F^*$-subscheme of the affine group $F^*$-scheme $G \times \Spec(F^*)$. For every non-empty  internal set of prime ideals $P$, we have $\pi_P(H(\prod_U \overline{F})) \subseteq H(\prod_U \overline{F})$.   If $H_n(F)$ satisfies Strong Approximation for every $n$ then so does $H(F^*)$, i.e., the map $\overline{H(F^*)}\rightarrow H(\overline{F^*})$ is an isomorphism of topological groups. Finally, if $P$ is a non-empty internal subset   and $H$ satisfies Strong Approximation then   $\pi_P(H(\prod_U \overline{F})) =H(\prod_U \overline{F}) \cap G(\prod_U \overline{F})_P $.
	
	\subsection{The commgruence completion}

\begin{definition} 	We  denote the commutator subgroup of $\SL_3(\Lambda^*)$ by $\Co(\Lambda^*)$.  For every non-zero ideal $\mathfrak{q} \lhd O^*$, we denote the $\mathfrak{q}$th congruence subgroup of $\SL_3(\Lambda ^*)$ by $\SL_3(\Lambda ^*; \mathfrak{q})$ and denote $\Co(\Lambda^*;\mathfrak{q}):= \Co(\Lambda^*) \cap \SL_3(\Lambda^*;\mathfrak{q})$. Every group which contains a group of the form $\Co(\Lambda^*;\mathfrak{q})$ for some non-zero  $\mathfrak{q} \lhd O^*$, is called a commgruence subgroup.  The commgruence topology is  coarsest group topology under which all commgruence subgroups are open. 
	\end{definition}

	Our next goal is to compute the commgruence completion $\widehat{\SL_3(\Lambda^*)}$ of $\SL_3(\Lambda^*)$.  
	By our assumption, $\SL_3(\Lambda)$ is perfect. Local computations show that  there exists a constant $c$ such that, for every non-zero ideal $\mathfrak{q}$ of $O$, $\SL_3(\Lambda)=[\SL_3(\Lambda),\SL_3(\Lambda)]^c\SL_3(\Lambda;\mathfrak{q})$. It follows that the same is true when $\SL_3(\Lambda)$ and $\mathfrak{q}$ are replaced by   $\SL_3(\Lambda^*)$ and any ideal in $\mathcal{I}$. In particular, for every $\mathfrak{q} \in \mathcal{I}$, 
	\[
	\SL_3(\Lambda^*)/\Co(\Lambda^*;\mathfrak{q})=\SL_3(\Lambda^*)/\Co(\Lambda^*) \times \SL_3(\Lambda^*)/\SL_3(\Lambda^*;\mathfrak{q}).
	\]
	
Taking inverse limit, we get
	\begin{lemma}\label{lemma:commgruence _completion }
		There exists an isomorphism of topological groups 
		\begin{equation}\label{eq:commgruence_completion}  
			\widehat{\SL_3(\Lambda^*)}=\SL_3(\Lambda^*)_{ab} \times \overline{\SL_3(\Lambda^*)},
		\end{equation}
		where $\SL_3(\Lambda^*)_{ab}$ is the abelization of $\SL_3(\Lambda ^*)$ with the discrete topology. 
	\end{lemma}

The next step is to show that the commgruence topology of $\SL_3(\Lambda^*)$ can be extended to a group  topology on $\SL_3(D^*)$. We will need the following:

\begin{definition} Let $A$ be an associative ring. For a natural number $n$ and an ideal $\mathfrak{q}$ of $A$, let \begin{enumerate}
\item $E_n(A) \subseteq \GL_n(A)$ be the subgroup generated by the set of elementary matrices $\{e_{i,j}(a)\mid 1 \le i \ne j \le n, a \in A\}$.
\item $E_n(A; \mathfrak{q})$ be the normal subgroup of $E_n(A)$ generated by all $E_n(A)$-conjugates of   the  set  $\{e_{i,j}(a)\mid 1 \le i \ne j \le n, a \in \mathfrak{q}\}$.
\end{enumerate} 
\end{definition} 

\begin{definition} The stable range of a unital ring $A$ is the infimum of the set of positive integers $n$ such that whenever $a_0,\ldots,a_n\in A$ generate the unit ideal, there are $b_1,\ldots,b_n\in A$ such that $a_1-b_1a_0,\ldots,a_n-b_na_0$ also generate the unit ideal.
\end{definition}

	\begin{theorem}[\cite{Bas}]\label{thm:Bass}
		\begin{enumerate}
			\item[] 
			\item\label{item:Ba2} (Item (f) of Theorem 4.2)  For every ring $A$ of stable range  $r$, every non-zero $\mathfrak{q} \lhd A$, and   every $n >\max(r,2)$,
			\[
			\El_n(A,\mathfrak{q})=\left\langle \left[\GL_n(A),\GL_n(A;\mathfrak{q})\right]  \right\rangle.
			\]			
			\item\label{item:Ba1} (Theorem 11.1)  The stable range of every order in a finite dimensional central division algebra over a number field is at most 2. 
			\item \label{item:Ba3}(Item (c) of Theorem 11.2) If $A$ is a division algebra and $n \ge 3$ then $\El_n(A)$ is the commutator subgroup of $\GL_n(A)$.
		\end{enumerate}
	\end{theorem}	
	
	\begin{corollary}\label{cor:Bass}
For every non-zero $\mathfrak{q} \lhd O^*$,
		\begin{equation}\label{eq:Bass}
			[\SL_3(\Lambda^*), \SL_3(\Lambda^*;\mathfrak{q}^2) ] \subseteq \El_3(\Lambda^*;\mathfrak{q}^2\Lambda^*)\subseteq \langle [\SL_3(\Lambda^*;\mathfrak{q}), \SL_3(\Lambda^*;\mathfrak{q}) ]\rangle.
		\end{equation}
	\end{corollary}
	\begin{proof}
		The stable range condition can be expressed in the first order language of rings. Thus, by Item \eqref{item:Ba1} of \ref{thm:Bass}, the stable range of $\Lambda^*$ is at most 2. By Item \eqref{item:Ba2} of \ref{thm:Bass}, for every non-zero $\mathfrak{q} \lhd O^*$,
		\[	
		\El_3(\Lambda^*,\mathfrak{q}\Lambda^*)=\left\langle \left[\GL_3(\Lambda^*),\GL_3(\Lambda^*;\mathfrak{q}\Lambda^*)\right]  \right\rangle
		\]
		and the left inclusion in Equation \eqref{eq:Bass} follows. The right inclusion holds since, for every $a,b \in \mathfrak{q}\Lambda^*$ and every distinct 
		$1 \le i,j,k \le 3$, $[e_{i,k}(a),e_{k,j}(b)]=e_{i,j}(ab)$.
	\end{proof}

	The following theorem was attributed to Draxl in \cite{DV88} but the reference given there does not seem to exist in Mathscinet.     
    %The following is proved in \cite{MoRa}. 
    The special case where the degree of the algebra is square free follows from the last corollary of \cite{Wan} and a compactness argument. 
	
	\begin{theorem}\label{thm:Draxl}
		Let $m \ge 1$  and $A$ a finite dimensional central division algebra over a number field $K$. There exists a constant $c$ such that every element of $\SL_m(A)$ is the product of $c$ commutators from $\GL_m(A)$.
	\end{theorem}
	
	\begin{lemma}\label{lemma:Bounded_over_D} The group $\SL_3(D^*)$ is boundedly generated by the elementary matrices. In particular,  
		$\SL_3(D^*)$ is perfect and normally generated by $\SL_3(\Lambda^*)$. 
	\end{lemma} 

\begin{proof} Theorem \ref{thm:Draxl} implies that every element of $\SL_3(D^*)$ is a product of a bounded number of commutators from $\GL_3(D^*)$. Item \eqref{item:Ba3} of Theorem \ref{thm:Bass} and a compactness argument imply that every commutator in $\GL_3(D^*)$ is a product of boundedly many elementary matrices, so the first claim holds.

Since every elementary matrix is the commutator of two elementary matrices, $\SL_3(D^*)$ is perfect.  Finally, every elementary matrix is conjugate to an elementary matrix in $\SL_3(\Lambda^*)$. 
\end{proof}	
	
	\begin{lemma}\label{lemma:com_topoolgy_extends}
		The commgruence topology  can be extended to a group topology on $\SL_3(D^*)$. 
	\end{lemma} 
	\begin{proof}
		It is enough to show that for every $g \in \SL_3(D^*)$ and every $\mathfrak{q} \in \mathcal{I}$ there exists a non-zero $\mathfrak{r} \in \mathcal{I}$ such that $g\Co(\Lambda^*;\mathfrak{r})g^{-1} \subseteq \Co(\Lambda^*;\mathfrak{q})$.  By Lemma \ref{lemma:Bounded_over_D}, every element of $\SL_3(D^*)$ is a product of a finite number of conjugates of elements from $\SL_3(\Lambda^*)$. Thus, by induction on this number, it is enough to prove the claim when $g$ is a conjugate of an element in $\SL_3(\Lambda^*)$. 
		
		Let $g \in \SL_3(D^*)$ and $h \in \SL_3(\Lambda^*)$.  
		Choose non-zero finitely generated ideals $\mathfrak{r},\mathfrak{s}\lhd O^*$ such that
		\begin{itemize}
			\item $g\SL_3(\Lambda^*;\mathfrak{s})g^{-1} \subseteq \SL_3[\Lambda^*;\mathfrak{q}]$.
			% 	\item $(\SL_3^*[;\Lambda^*\mathfrak{r}_1])_w$$ [\Gamma^*[J],\Gamma^*[I_2]]$ ($I_2$ exists by Lemma \ref{lemma:width_assum}\ref{item:width_assum3}).
			\item $g^{-1}\SL_3(\Lambda^*;\mathfrak{r})g \subseteq \SL_3[\Lambda^*;\mathfrak{s}^2]$ and $\mathfrak{r} \subseteq \mathfrak{q}$. 
		\end{itemize}
		Corollary \ref{cor:Bass} implies that  $[\SL_3(\Lambda^*),\SL_3(\Lambda^*;\mathfrak{s}^2)] \subseteq \langle [\SL_3(\Lambda;\mathfrak{s}), \SL_3(\Lambda;\mathfrak{s})]\rangle$. Thus, for every
		$f \in \Co(\Lambda^*;\mathfrak{r})$,
		$$
		g(hg^{-1}fgh^{-1})g^{-1}=g[h,g^{-1}fg]g^{-1}f \in  g[\SL_3(\Lambda^*),\SL_3(\Lambda^*;\mathfrak{s}^2)]g^{-1}f \subseteq  g \langle [\SL_3(\Lambda;\mathfrak{s}), \SL_3(\Lambda;\mathfrak{s})]\rangle g^{-1}f\subseteq  \Co(\Lambda^*;\mathfrak{q}).$$
	\end{proof}

	\section{Discrete metaplectic extensions}

	%\subsection{Construction of a discrete metaplectic extension}

	\begin{definition}
		A discrete metaplectic extension of $\SL_3(D^*)$ is a topological central extension 
		\[
		1 \rightarrow C \rightarrow E \overset{\rho}\rightarrow  \overline{\SL_3(D^*)} \rightarrow 1.
		\]
		such that 
		\begin{enumerate}
			\item The topology on $C$ is discrete.
			\item $\rho$ has  a section over $\SL_3(D^*)$.
			\item $\rho$ has  a continuous  section over $\SL_3(\overline{\Lambda^*})$.
		\end{enumerate}
	\end{definition}
	
	The goal of this section is to show that, under the assumptions of Theorem \ref{thm:main}, if $\SL_3(\Lambda)$ is not boundedly generated by elementary matrices, then there exists a discrete metaplectic extension of $\SL_3(D^*)$ which does not have a continuous section. 
	
	By the assumptions of Theorem \ref{thm:main}, $\SL_3(\Lambda)=\El_3(\Lambda)$ is perfect.  Item \ref{item:Ba2} of Theorem \ref{thm:Bass} and a standard compactness argument show that every element in the set $[\SL_3(\Lambda),\SL_3(\Lambda)]$ of commutators is a product of a bounded number of elementary matrices. On the other hand, every elementary matrix is the commutator of two elementary matrices.   Thus, every element in $\SL_3(\Lambda)$ is the product of boundedly many elementary matrices if and only if every element in $\SL_3(\Lambda)$ is the product of boundedly many elements from $[\SL_3(\Lambda),\SL_3(\Lambda)]$.    A standard model theoretic  argument shows that the following properties are equivalent: 
	\begin{enumerate}[(1)]
		\item  Every element in $\SL_3(\Lambda)$ is the product of boundedly many elements from $[\SL_3(\Lambda),\SL_3(\Lambda)]$. 
		\item The abelianization  $\SL_3(\Lambda^*)_{ab}$ of  $\SL_3(\Lambda^*)$ is trivial.   
		\item The abelianization  $\SL_3(\Lambda^*)_{ab}$ of  $\SL_3(\Lambda^*)$ is finite.   
	\end{enumerate}

	By Lemma \ref{lemma:commgruence _completion }, there exists a central topological extension  
	\begin{equation}\label{eq:integral_extesion}
		1 \rightarrow \SL_3(\Lambda^*)_{ab} \rightarrow \widehat{\SL_3(\Lambda^*)}=\SL_3(\Lambda^*)_{ab} \times \overline{\SL_3(\Lambda^*)}\rightarrow  \overline{\SL_3(\Lambda^*)} \rightarrow 1
	\end{equation}
	where the abelianization $\SL_3(\Lambda^*)_{ab}$ of $\SL_3(\Lambda^*)$ is given the discrete topology. 

By Lemma \ref{lemma:com_topoolgy_extends}, the commgruence topology extends to $\SL_3(D^*)$. Since the commgruence topology of $\SL_3(\Lambda ^*)$ has a completion, so does the commgruence topology on $\SL_3(D^*)$. Denoting the completion by $\widehat{\SL_3(D^*)}$, we get a short exact sequence
\begin{equation}\label{eq:rational_extesion.1}
		1 \rightarrow C \rightarrow \widehat{\SL_3(D^*)}\overset{\rho}\rightarrow  \overline{\SL_3(D^*)} \rightarrow 1.
	\end{equation}
Since $\widehat{\SL_3(\Lambda ^*)}$ is open in $\widehat{\SL_3(D^*)}$, the kernel $C$ is $\SL_3(\Lambda ^*)_{ab}$. By Lemma \ref{lemma:Bounded_over_D}, $C$ is central in $\widehat{\SL_3(D^*)}$. Thus, we get a discrete metaplectic extension
\begin{equation}\label{eq:rational_extesion}
		1 \rightarrow \SL_3(\Lambda^*)_{ab} \rightarrow \widehat{\SL_3(D^*)}\overset{\rho}\rightarrow  \overline{\SL_3(D^*)} \rightarrow 1.
	\end{equation}
    \begin{lemma}\label{lemma:dis_meta_exten_no_cont_sec}
		If $\SL_3(\Lambda)$ is not boundedly generated by elementary matrices, then there exists a discrete metaplectic extension of $\SL_3(D^*)$ which does not have a continuous section. 
	\end{lemma}
	\begin{proof}
		Assume that  extension \eqref{eq:rational_extesion} 
		has a continuous section $\eta:  \overline{\SL_3(D^*)} \rightarrow \widehat{\SL_3(D^*)}$.  The map $\chi:\widehat{\SL_3(D^*)}\rightarrow C$ defined by $\chi(g):=g^{-1}\eta( \rho(g))$
		is a continuous epimorphism from $\widehat{\SL_3(D^*)}$ onto the abelian group $\SL_3(\Lambda^*)_{ab}$. By Lemma \ref{lemma:Bounded_over_D}, the dense subgroup $\SL_3(D^*)$ of $\widehat{\SL_3(D^*)}$ is prefect. Thus, $\SL_3(\Lambda^*)_{ab}=\chi(\widehat{\SL_3(D^*)})=1$. 
	\end{proof}

\section{Proof of Theorem \ref{thm:main} }
{\bf Step A:} View $\SL_3(\Lambda)$ as a subgroup of $\SL_d(\Lambda)$ via the embedding that sends a matrix $A \in \SL_3(\Lambda)$ to the diagonal matrix  $\diag(A,1,\ldots,1)$.
    Since $\Lambda$ has stable range 2, every element of $\SL_d(\Lambda)$ can be transformed to an element of $\SL_3(\Lambda)$ by a multiplication with a bounded number of elementary matrices. Thus, it is enough to prove Theorem \ref{thm:main} for $d=3$. By Lemma \ref{lemma:dis_meta_exten_no_cont_sec}, in order to prove Theorem \ref{thm:main} it is enough to show that every discrete metaplectic extension of $\SL_3(D^*)$ has a continuous section. 
	Let 
	\[
	1 \rightarrow C \rightarrow E \overset{\rho}\rightarrow \SL_3(\overline{D^*})\rightarrow 1
	\]
	be a discrete metaplectic extension and let $\eta:\SL_3(\overline{\Lambda^*})\rightarrow E$ be a continuous section  over $\SL_3(\overline{\Lambda^*})$. \\\\
	{\bf Step B:}  The topological isomorphism $\psi: \overline{D}\rightarrow M_n(\overline{F})$ is a linear map defined over $\overline{F}$ so it induces a $\prod_U \overline{F}$-linear topological  isomorphism $\psi:\prod_U \overline{D} \rightarrow \prod_U M_n(\overline{F})$ where the topologies on the domain and range are induced from the embeddings of  $\prod_U \overline{D}$ and $\prod_U M_n(\overline{F})$  in  $\overline{D^*}$ and  $M_n(\overline{F^*})$.  By the universal property of completions, $\psi $ has a unique extension to a topological isomorphism $\psi:\overline{D^*}\rightarrow M_n(\overline{F^*})$. 
	Since $\psi(\overline{\Lambda})=  M_n(\overline{O})$, 
	$\psi(\prod_U \overline{\Lambda})=  \prod_U M_n(\overline{O})$ and $\psi(\overline{\Lambda^*})=M_n(\overline{O^*})$.
	Moreover, $\psi$ induces  a  topological isomorphism $\psi:{\SL_3(\overline{D^*})}\rightarrow \SL_{3n}(\overline{F^*})$ such that $\psi(\SL_3(\overline{\Lambda^*})
	)=\SL_{3n}(\overline{O^*})$. 
	\\\\
	{\bf Step C:}
	We get a topological central extension 
	\begin{equation}\label{eq:DME2}\tag{*}
		1 \rightarrow C \rightarrow E \overset{\psi \circ \rho}\rightarrow \SL_{3n}(\overline{F^*})\rightarrow 1
	\end{equation}
	with a continuous section $\eta\circ  \psi^{-1} $ over $\SL_{3n}(\overline{O^*})$ and a possibly discontinuous section over $\psi( \SL_3(D^*))$.  In order to prove Theorem \ref{thm:main}, it is enough to prove that Extension \eqref{eq:DME2} has a continuous section. In this step we show that in order to prove that Extension \eqref{eq:DME2} has a continuous section, it is enough to prove that it has a possibly discontinuous section over $\SL_{3n}(F^*)$. 
	%Note that  $\SL_{3n}(F^*) \cap \SL_{3n}(\overline{O^*})=\SL_{3n}(O^*)$. 
	
   Indeed, since $S$ contains a finite place, the group $\SL_{3n}(O)$ has the Congruence Subgroup Property in the strict sense so it is boundedly generated by elementary matrices.   It follows that every element of $\SL_3(O)$ is a product of a bounded number of commutators, so $\SL_3(O^*)$ is perfect.  If $\eta'$ is a possibly discontinuous section of \eqref{eq:DME2} over $\SL_{3n}(F^*)$ then $\eta \circ \psi^{-1}$ and $\eta'$ agree on the open subgroup $\SL_{3n}(F^*) \cap \SL_{3n}(\overline{O^*})=\SL_{3n}(O^*)$ of $\SL_{3n}(F^*)$. It follows that $\eta'$ is continuous so it extends to a continuous section of \eqref{eq:DME2}. 
	\\\\
	\noindent {\bf Step D:} For every $a,b \in (F^*)^\times$, denote   $h_1(a):=\diag(a,a^{-1},1,\dots,1)\in \SL_{3n}(F^*)$ and  $h_2(b):=\diag(b,1,b^{-1},1,\dots,1)\in \SL_{3n}(F^*)$. For every two commuting elements $g_1,g_2 \in \SL_{3n}(\overline{F^*})$, let $[g_1:g_2]\in E$ be 
	the commutator of some lifts of $g_1$ and $g_2$ to $E$. The element $[g_1:g_2]$ is independent of the choice of lifts. A Theorem of Steinberg implies that the extension \eqref{eq:DME2} has  a section  over $\SL_{3n}(F^*)$ if and only if for every $a,b \in (F^*)^\times$, $[h_1(a):h_2(b)]=1$.
	
	Let  $a=[a_n]_n,b=[b_n]_n \in (F^*)^\times$.  We can assume that for every $n$, $a_n,b_n \in F^\times$.  For every $n$, let
	$P_n$ be the set of primes ideals $\mathfrak{p}$ of $O$ such that $\val_{\mathfrak{p}}(a_n)\ne 0$ or $\val_{\mathfrak{p}}(b_n)\ne 0$.
	Then $P=[P_n]_n$ is small internal. A standard argument, see the proof of Item 2 of Proposition 6.3.4  of \cite{AM}, shows that \[[h_1(a):h_2(b)]=[\pi_P(h_1(a)):\pi_P(h_2(b))] \cdot [\pi_{P^c}(h_1(a)):\pi_{P^c}(h_2(b))].\]
	Since $\pi_{P^c}(h_1(a)),\pi_{P^c}(h_2(b))\in \SL_{3n}(\prod_U \overline{O})\subseteq  \SL_{3n}(\overline{O^*})$, 
	\[
	[\pi_{P^c}(h_1(a)):\pi_{P^c}(h_2(b))]= \Big[\eta \circ \psi^{-1}(\pi_{P^c}(h_1(a)))\, ,\, \eta \circ \psi^{-1}(\pi_{P^c}(h_2(b)))\Big]=1,
	\]	
	so
	\[[h_1(a):h_2(b)]=[\pi_P(h_1(a)):\pi_P(h_2(b))].\] 
	Since $C$ is discrete, for every pair $h_1,h_2 \in \SL_{3n}(\overline{F^*})$ of commuting elements  there exists a non-zero finitely generated ideal $\mathfrak{q} \lhd O^*$ such that if $g_1\in h_1\SL_{3n}(\overline{O^*},\bar{\mathfrak{q}})$ and $g_2 \in h_2\SL_{3n}(\overline{O^*},\bar{\mathfrak{q}})$ are commuting elements then $[g_1:g_2]=[h_1:h_2]$.   By the centrality of $C$,  for every two commuting elements $g_1,g_2 \in \SL_{3n}(\overline{F^*})$ and every $\ell \in \SL_{3n}(\overline{F^*})$, $[\ell g_1\ell^{-1}:\ell  g_2  \ell^{-1}]=[g_1:g_2]$. 
	%For every $\ell \in \SL_{3n}(\overline{F^*})$ and every commuting $g_1,g_2 \in \in \SL_{3n}(\overline{F^*})$, $[g_1:g_2]=[\ell g_1 \ell^{-1},\ell g_2 \ell^{-1}]$.
	Thus, by Step (C), Theorem \ref{thm:main} follows from:
	
	\begin{claim}\label{claim:Step_4} For every $a,b\in F^*$, every small internal subset $P$ and every non-zero finitely generated ideal $\mathfrak{q} \lhd O^*$, there exist a commuting pair $g_1,g_2 \in \SL_{3n}(\overline{F^*})$ and an element $\ell \in \SL_{3n}(\overline{F^*})$ such that 
		$\pi_P(\ell g_1 \ell^{-1}) \in \pi_P(h_1(a))\SL_{3n}(\overline{O^*},\overline{\mathfrak{q}})$, $\pi_P(\ell g_2\ell^{-1}) \in \pi_P(h_2(b))\SL_{3n}(\overline{O^*},\overline{\mathfrak{q}})$, and $[\pi_P(g_1):\pi_P(g_2)]=1$.
		
	\end{claim}		 
	\noindent
	{\bf Step E:} The goal of this step is to reduce Claim \ref{claim:Step_4} to  Claim \ref{claim:Step_5} below.

	\begin{lemma}\label{lemma:Pra_Rag}
		Let $c,d \in F^\times$, let $T$ be a finite set of places of $F$ which is disjoint from $S$, and, for every $v \in T$, let $n_v \in \N$. There exist a maximal subfield $K$ of $D$, two commuting elements $g_1,g_2 \in \SL_3(K)$, and an element $\ell \in \SL_3(D)$ such that for every $v \in T$, $\pi_v(\psi(\ell g_1 \ell^{-1})) \in \pi_v(h_1(c))\SL_{3n}(O_v; \mathfrak{p}_v^{n_v}O_v) $ and  $\pi_v(\psi(\ell g_2 \ell^{-1})) \in \pi_v(h_1(d))\SL_{3n}(O_v; \mathfrak{p}_v^{n_v}O_v)$. 
	\end{lemma}
	\begin{proof}%[Sketch of proof for Lemma  \ref{lemma:Pra_Rag}]
		%In the proof of  Theorem 4.2 of \cite{PrRa} ,Prasad and Raghunathan proved the statement of the lemma when $T$ is a singleton. A similar proof works  when $T$ is finite. 
		
		The proof is very similar to \cite{PrRa}. By \cite[proof of Theorem 4.2]{PrRa}, there is a maximal subfield $K \subseteq D$ and, for each $v\in T$, an element $h_v\in \left( D \otimes F_v \right) ^ \times$ such that $\psi \left( h_v K_v h_v ^{-1} \right)$ is the set of diagonal matrices in $M_n(F_v)$. Thus, by weak approximation, there are  elements $x_a,x_b\in K$ such that
		\[
		\psi \left( h_v x_a h_v ^{-1} \right) \in \left( \begin{matrix} a & & & \\ & 1 & & \\ & & \ddots & \\ & & & 1 \end{matrix} \right)\cdot \GL_n \left( O_v,\mathfrak{p}_v^{n_v} \right), \quad \psi \left( h_v x_b h_v ^{-1} \right) \in \left( \begin{matrix} b & & & \\ & 1 & & \\ & & \ddots & \\ & & & 1 \end{matrix} \right)\cdot \GL_n \left( O_v,\mathfrak{p}_v^{n_v} \right)
		\]
		\[
		\psi \left( h_v x_a^{-1} h_v ^{-1} \right) \in \left( \begin{matrix} a^{-1} & & & \\ & 1 & & \\ & & \ddots & \\ & & & 1 \end{matrix} \right)\cdot \GL_n \left( O_v,\mathfrak{p}_v^{n_v} \right) \text{ and } \psi \left( h_v x_b^{-1} h_v ^{-1} \right) \in \left( \begin{matrix} b^{-1} & & & \\ & 1 & & \\ & & \ddots & \\ & & & 1 \end{matrix} \right)\cdot \GL_n \left( O_v,\mathfrak{p}_v^{n_v} \right),
		\]
		for all $v\in T$. If $P\in \SL_{3n}(F)$ is the permutation matrix of the involution $(2,n+1)(3,2n+1)$, then
		\[
		P \psi \left( \left( \begin{matrix} h_v & & \\ & h_v & \\ & & h_v ^{-2} \end{matrix} \right) \left( \begin{matrix} x_a & 0 & 0 \\ 0 & x_a ^{-1} & 0 \\ 0 & 0 & 1 \end{matrix} \right) \left( \begin{matrix} h_v ^{-1} & & \\ & h_v ^{-1} & \\ & & h_v ^{2} \end{matrix} \right) \right) P^{-1} \in h_1(a) \SL_{3n}(O_v; \mathfrak{p}_v^{n_v})
		\]
		and
		\[
		P \psi \left( \left( \begin{matrix} h_v & & \\ & h_v & \\ & & h_v ^{-2} \end{matrix} \right) \left( \begin{matrix} 1 & 0 & 0 \\ 0 & x_b & 0 \\ 0 & 0 & x_b ^{-1} \end{matrix} \right) \left( \begin{matrix} h_v ^{-1} & & \\ & h_v ^{-1} & \\ & & h_v ^{2} \end{matrix} \right) \right) P^{-1} \in h_2(b) \SL_{3n}(O_v; \mathfrak{p}_v^{n_v}),
		\]
		so we can take $g_1:=\left( \begin{matrix} x_a & 0 & 0 \\ 0 & x_a ^{-1} & 0 \\ 0 & 0 & 1 \end{matrix} \right) \in \SL_3(K)$, $g_2:=\left( \begin{matrix} 1 & 0 & 0 \\ 0 & x_b & 0 \\ 0 & 0 & x_b ^{-1} \end{matrix} \right) \in \SL_3(K)$, and $\ell$ as any element of $\SL_3(D)$ that approximates $\psi ^{-1} (P) \left( \begin{matrix} h_v & & \\ & h_v & \\ & & h_v ^{-2} \end{matrix} \right)$ for every $v\in T$.
	\end{proof}
	Let $a,b\in F^*$, let $P$ be a small internal set and let $\mathfrak{q} \lhd O^*$ be a non-zero finitely generated ideal  of $O^*$.  
	Lemma \ref{lemma:Pra_Rag} implies that  there exist a sequence $(K_n)_n$ of maximal subfields of $D$, two commuting elements $g_1,g_2 \in \SL_3(\prod_U K_n)$, and an element $\ell \in \SL_{3}(D^*)$ such that $\pi_P( \psi(\ell g_1 \ell^{-1})) \in \pi_P(h_1(a))\SL_{3n}(\overline{O^*},{\mathfrak{q}}\overline{O^*})$ and  $\pi_P(\psi(\ell g_2 \ell^{-1})) \in \pi_P(h_2(b))\SL_{3n}(\overline{O^*},{\mathfrak{q}}\overline{O^*})$.  Thus, in order to prove Claim \ref{claim:Step_4}, it is enough to show that 
	$[\pi_P(\psi(g_1)):\pi_P(\psi( g_2))]=1$.
	\begin{enumerate}[(a)]
		\item\label{item:Step_5a} 	The elements $\pi_P(\psi(g_1))$ and $\pi_P(\psi( g_2))$ are commuting elements which belongs to $\psi(\SL_3(\prod_U \overline{K_n}))$. By the last paragraph of \S\ref{sec:the_congruence_completion},  $\psi(\SL_3(\prod_U \overline{K_n}))$ is contained in  the closure of $\psi(\SL_3(\prod_U {K_n}))$. 
		Hence, in order to prove that $[\pi_P(\psi(g_1)):\pi_P(\psi( g_2))]=1$, it is enough to show that there exists a  section of Extension \eqref{eq:DME2} over the closure of $\psi(\SL_3(\prod_U {K_n}))$.
		\item\label{item:Step_5b}  Let $\xi$ be a section of Extension \eqref{eq:DME2} over $\psi(\SL_3(D^*))$ and recall that  $\eta\circ \psi^{-1}$ is a continuous section of Extension \eqref{eq:DME2} over $\psi(\SL_{3}(\overline{\Lambda^*}))$.		 If $\xi$ and $\eta \circ \psi^{-1}$  agree over some open subgroup of $\psi(\SL_3(\prod_U {K_n}))$ then $\xi$ can be extended to a continuous section over the closure of $\psi(\SL_3(\prod_U{K_n}))$. 
		\item\label{item:Step_5c}  $\xi$ and $\eta \circ 
		\psi^{-1}$ agree over the commutator subgroup of $\psi(\SL_3(\prod_U K_n)) \cap \psi(\SL_{3}(\Lambda^*))=\psi(\SL_3(\prod_U K_n)) \cap \SL_3(\overline{\Lambda^*}))$. Thus, in order to finish the proof it is enough to show that the commutator subgroup of $\SL_3(\prod_U K_n)) \cap \SL_3(\overline{\Lambda^*})$ is an open subgroup of $\SL_3(\prod_U K_n))$.
	\end{enumerate}
	
	Items \ref{item:Step_5a}, \ref{item:Step_5b} and \ref{item:Step_5c} imply that   $[\pi_P(\psi(g_1)):\pi_P(\psi( g_2))]=1$ if the following claim holds: 
	\begin{claim}\label{claim:Step_5}
		Let $(K_n)_n$ be a sequence of maximal subfields of $D$ and $U$ an open subgroup of  $\SL_3(\prod_U K_n)$ in the $S$-congruence topology. The commutator subgroup of $U$ is an open subgroup of  $\SL_3(\prod_U K_n)$.
	\end{claim}	
	
	\noindent {\bf Step F:} The goal of this step is to prove Claim \ref{claim:Step_5} and thus to complete the proof of Theorem \ref{thm:main}.  In order to prove Claim  \ref{claim:Step_5}, it is enough to show that there exists a constant $c$ such that for every maximal subfield $K$ of $D$ and every  open subgroup $V$ of $\SL_{3n}(K)$ (with respect to the $S$-congruence topology), there exists an open subgroup $W \le V$ such that every element of $W$ is the product of at most $c$ elements from $[V,V]$.   Corollary 3.13 of  \cite{Mor}  implies that there exists a constant $c$ such that for every maximal subfield $K$ of $D$ and every open subgroup $V$ of $\SL_{3n}(K)$, there exists a finite index subgroup $W $ of $  V$ such that every element of $W$ is the product of at most $c$ elements from $[V,V]$.   
	Since $S$ contains a finite place, $\SL_3(O_K)$ has the Congruence Subgroup Property in the strict sense so $W$ must be open. The proofs of Claim \ref{claim:Step_5} and  Theorem \ref{thm:main} is now complete.

\end{document}